\RequirePackage{fix-cm}
\documentclass[smallextended]{svjour3}       
\smartqed  
\usepackage{subfigure}
\usepackage{amsmath}
\usepackage{amsfonts}
\usepackage{amssymb}
\usepackage{color}
\usepackage{fixltx2e}
\usepackage{algorithm}
\usepackage{algorithmic}
\usepackage{graphicx}
\usepackage[compress]{cite}
\usepackage{bbm}
\usepackage{epstopdf}
\usepackage{enumerate}

\usepackage{mathptmx}      
\newcommand{\crb}{\color{black}}
\newcommand{\crg}{\color{black}}

\newcommand{\crbb}{\color{black}}

\newcommand{\mz}{\color{black}}

\begin{document}

\title{Constrained Optimal Consensus in Dynamical Networks}



\author{Amir Adibzadeh \and
        Mohsen Zamani \and
        Amir A. Suratgar \and
        Mohammad B. Menhaj
}


\institute{Amir Adibzadeh \at
              Electrical Engineering Department, Amirkabir University of Technology, Tehran 15914, Iran\\
              \email{amir.adibzadeh@aut.ac.ir}           
           \and
           Mohsen Zamani \at
           School of Electrical Engineering and Computer Science, The University of Newcastle, Callaghan, NSW 2308, Australia\\
           \email{mohsen.zamani@newcastle.edu.au}
           \and
           Amir A. Suratgar \at
             Electrical Engineering 	Department, Amirkabir University of Technology, Tehran 15914, Iran\\
             \email{a-suratgar@aut.ac.ir}
           \and
           Mohammad B. Menhaj \at
           Electrical Engineering 	Department, Amirkabir University of Technology, Tehran 15914, Iran\\
           \email{menhaj@aut.ac.ir}
}

\date{Received: date / Accepted: date}

\maketitle

\begin{abstract}
In this paper, an optimal consensus problem with local inequality constraints
is studied for a network of single-integrator agents. The goal is that a group of single-integrator agents
rendezvous at the optimal point of the sum of local convex objective functions.
The local objective functions are only available
to the corresponding agents {\crbb that only need to know their relative distances from their neighbors
in order to seek the final optimal point.} This point is supposed
to be confined by some local inequality constraints. To tackle this problem, we integrate the primal-dual
gradient-based optimization algorithm with a consensus protocol to
drive the agents toward the agreed point that satisfies KKT conditions.
The asymptotic convergence of the solution of the optimization problem
is proven with the help of LaSalle's invariance principle for hybrid
systems. A numerical example is presented to show the effectiveness of
our protocol.
\end{abstract}

\section{Introduction}
\label{intro}
Over the last decade, cooperative control in a network of autonomous
agents have been considered in scientific communities by virtue of big
breakthroughs in wireless communication technology. Among these problems, consensus in dynamical networks is
a central problem that has been studied from many aspects \cite{ren2007distributed,cheng2016reaching,fan2014semi,rezaee2015average}.
In particular, the problem of optimal consensus
among networked agents has recently gained considerable attention.
In this setup, the final consensus value is required to minimize the
sum of individual uncoupled convex functions.
For instance, the paper \cite{rahili2015distributed} resolved
the optimal consensus problem over a network of single-integrator agents with time-varying objective function
under the confining condition that Hessians associated with all local convex functions being identical.
Later, they suggested a more sophisticated algorithm to relax
this restriction.
The authors of \cite{xie2017global}
proposed a bounded control law applied to a network of single-integrator agents to resolve
the similar problem.
In these works, agents admit no constraint.
The optimal consensus problem can be formulated as distributed optimization
problem \cite{nedic2010constrained,lin2014constrained,lu2012zero}.
In this setup, all interconnected agents cooperate with each other to seek
the global optimal solution in a cooperative manner. Each agent minimizes its local
cost function and exchanges its states' information with its neighbors so that the team achieves the global optimum solution.
In \cite{nedic2010constrained}, a consensus
protocol is integrated with a projection operator to reach an agreed point that is limited to
the intersection of local constraint sets to solve a
distributed constrained optimization problem. The article \cite{lin2014constrained}
extended the work by \cite{nedic2010constrained} to study the problem
of constrained consensus in unbalanced networks.
The authors in \cite{lu2012zero} presented
a set of continuous-time algorithms called Zero-Gradient-Sum, by which
the states of a whole network asymptotically converge
 to the solution to the associated unconstrained convex optimization
problem along an invariant zero-gradient-sum manifold.\\
To resolve distributed optimization problems with inequality/equality constraints,
some researches were conducted based on primal-dual methods.
For example, 
the reference \cite{yi2015distributed} proposed a continuous-time dynamics
for seeking the saddle
point of the sum of agents' local Lagrangians to solve a distributed optimization problem
with both inequality and equality constraints per node.
The research paper  {\cite{yang2016multi} presented a continuous-time
protocol for distributed optimization problems with general constraints, relaxing
the assumption of global convexity on each local objective function
to convexity of locally bounded feasible region.
{\crbb In the both above-mentioned works, to attain consensus
on the globally optimal solution, a Lagrangian multiplier is assigned to each agent
for accommodating to consensus equality constraints.
Then, all agents exchange the information of their Lagrangian multipliers (dual states) as well as the primal decision variables 
with their neighbors in order to
reach consensus on the optimal solution. The papers \cite{wang2011control,gharesifard2014distributed,niederlander2016distributed} exploited the same technique to fulfill consensus over networks.}\\
{\crbb In particular, in this paper, a novel solution to optimal consensus
problems over undirected networks of
single-integrator agents is presented. Such solution must also satisfy
local convex inequalities for all agents. To tackle this problem, we split it into
two parts, namely, a consensus sub-problem and an optimization one.
Following this segmentation, we then propose a distributed
continuous-time solution that consists of two parts:
the first part yields the optimal points associated with local cost functions, and,
at the same time, the second part drives the agents toward reaching consensus.
In the proposed algorithm, each agent only needs to know its relative distances from its neighbors as well as
 its own objective function and constraint information.
 It is worthwhile noting that in applications such as swarm robots, 
 the relative positions with their neighbors  might be  the only 
 information that agents can have access to for constructing proper control actions.
 Our proposed algorithm makes the communication establishment,
 which is essential in the literature for exchanging the information 
 of Lagrangian multipliers for reaching optimal consensus value, redundant and unnecessary.
 Besides, by the present approach less communication burdens are imposed on the agents
 in communication-based control setups.}
 To establish the proposed algorithm, we present the stability analysis
associated with perturbed dynamical systems and introduce a novel convergence proof with the help of LaSalle's invariance principle for hybrid systems.\\
The results of the current paper provides further developments compared to the existing literature in this area.\\
i) {\crbb Compared to \cite{wang2011control,gharesifard2014distributed,yang2016multi}, in the present approach the agents do not need to exchange the information of their dual variables and can reach optimal consensus by only knowing their relative positions with respect to their neighbors.}\\
 ii) From design perspective, the penalty-based protocol studied in \cite{wang2011control,gharesifard2014distributed,kia2015distributed,yi2015distributed,yang2016multi,niederlander2016distributed} only admits linear consensus paradigm. This restricts the protocol illustrated in these references from adopting nonlinear consensus strategies that can in turn deliver fast convergence outcomes, see e.g. \cite{rahili2015distributed}. {\crb Besides, in the case of high order dynamics, this approach does not work, see e.g. \cite{rahili2015distributed,xie2017global}.} The algorithm introduced here does not have such limitation.\\
  iii) Even though the problem studied here is closely related to that of  in \cite{rahili2015distributed,xie2017global,lu2012zero,kia2015distributed}, unlike the current paper, these references only addressed unconstrained optimization. \\
  iv) While the references \cite{nedic2010constrained,lin2014constrained}  explored constrained optimization problems with convex set constraints, the projection operator utilized therein is difficult to handle in real-time specially when a large number of constraints are involved. Since  a closed convex set can be approximated by  a polyhedron set that is constituted by a set of linear equalities and inequalities, one can cast the optimization problem of \cite{nedic2010constrained,lin2014constrained} into the present formulation and adopt an easy-to-handle gradient-based primal-dual method discussed here to resolve it.\\
  v) {\crb The proposed algorithm achieves a less perplexed states's trajectories toward the final point compared to the existing penalty-based algorithms (see Section \ref{sec:SE}).}}\\
\indent {This paper is organized as follows. 
The problem formulation is given Section \ref{sec:1}.
 Then, our proposed solution is presented
 in Section \ref{sec:MR}. \textcolor{black}{ A numerical example is presented in Section \ref{sec:SE}.}
Finally, the concluding remarks and suggestions for future studies are given in Section \ref{sec:Conc}.}
\section{Problem Formulation}
\label{sec:1}
{\crbb Consider $N$ physical agents over a network with time-invariant
undirected graph $\mathcal{G}=(\mathcal{N},E,A)$,
where $\mathcal{N}=\{1,\ldots,N\}$ is the node set, $E\subseteq \mathcal{N}\times\mathcal{N}$
is the edge set , and $A=[a_{ij}]$ is a weighted adjacency matrix.
Each pair $e=(i,j)\in E$ indicates link between the
node $i$ and the node $j$ in an undirected graph.}
Suppose that each agent is described by
the continuous-time single-integrator dynamics

\begin{equation}
\ensuremath{{\dot{x}_{i}}(t)={u_{i}}(t)},\quad\label{eq:1}
\end{equation}
where $x_{i}(t)\in\mathbb{R}$ represents the position of agent $i$,
and $u_{i}(t)$ is the control input to agent $i$. {\mz  We shall drop the argument $t$ throughout this paper unless it is necessary}. { \mz It is worhwhile noting that here} we consider only one dimensional agents for the sake of simplicity in notations.
However, it is straightforward to show that our algorithm can be extended to higher dimensional dynamics as each dimension is decoupled from others and can be treated independently. The agents are supposed to reach at an agreed point that shall minimize  a convex
optimization problem as
\begin{equation}
\begin{split}\min \,f(x)=\sum_{i=1}^{N}f_{i}(x)\\
\text{subject to}\,g_{i}(x)\leq0,
\end{split}
\label{eq:2}
\end{equation}
in which $f_{i}(\cdot):\mathbb{R\rightarrow\mathbb{R}}$
is the local cost function associated with node $i$ in the network.
Furthermore, $g_{i}(\cdot):\mathbb{R\rightarrow\mathbb{R}}$ {represents a constraint on the
optimal position and is associated with node $i$.
{ It is supposed that each agent knows only its associated cost function and inequality constraint function.}
 {\crg We assumed only one inequality constraint per node for complexity avoidance; our algorithm can solve the same problem with a desired number of inequality constraints.}

We consider the following assumptions in relation to the problem} (\ref{eq:2}).\\
\noindent \textit{Assumption 1.}
\begin{enumerate}[(i)]

    \item The objective functions $f_{i}(\cdot)$,
    {\mz $i\in\mathcal{N}$} , and $g_{i}(\cdot)$, {\mz $i\in\mathcal{N}$,} are convex and continuously differentiable
    on $\mathbb{R}$.
    \item The aggregate cost function $\sum_{i=1}^{N}f_{i}(\cdot)$
  is radially unbounded on $\mathbb{R}$.
\end{enumerate}
\textit{Assumption 2.} (Slater's Condition) There exists $x^{*}\in\mathbb{R}$ such that $g_{i}(x^{*})\leq0$.

The Assumption 1 and 2 fulfill the solution existence conditions for the optimization problem \eqref{eq:2}.
{Note that the constrained optimal consensus problem that was defined above is equivalent with
the following distributed convex optimization problem}
\begin{equation}
\begin{array}{cc}
\underset{\underset{i=1,\cdots,N}{x_{i}}}{\text{min}}\sum_{i=1}^{N}f_{i}(x_{i})\\
\text{subject to}\,x_{i}=x_{j}\,\text{and}\,g_{i}(x_{i})\leq0,
\end{array}
\label{eq:3}
\end{equation}
In the problem (\ref{eq:3}),
the consensus constraint, i.e. $x_{i}=x_{j},$ $i,j=1,\ldots,N$,
is imposed to guarantee the same decision variable is achieved eventually. Here,
the agents with dynamics as in \eqref{eq:1} shall seek {\mz the optimal point $x^*$ i.e. } $x_{i}=x^{*}$, $i\in \mathcal{N}$,
which minimize the collective cost function $\sum_{i=1}^{N}f_{i}(x_i)$
in a distributed fashion, given inequality constraints
$g_{i}(x_i)\leq0$, $i\in\mathcal{N}$. To this end, each agent searches for the minimum of
its associated  cost function, $f_{i}(x_{i})$, with regards to its local inequality constraint, $g_{i}(x_{i})$, not knowing other local
cost functions and constraint inequalities. Furthermore, all agents shall reach an agreement on their positions through
{\crbb only knowing the relative distances from their neighboring agents.} Now, we shall design
the control input $u_{i}$ to fulfill these requirements.

{\crg One can say that the problem \eqref{eq:3} consists of a minimization
{\mz sub-problem}, with inequality constraints, and a consensus {\mz sub-problem}.
 This splitting is the cornerstone of our approach
to resolve the problem \eqref{eq:2}.} The
minimization {\mz sub-problem} can be defined as

\begin{equation}
\begin{split}
\underset{\underset{ i=1,\cdots,N}{x_{i}}}{\min }\sum_{i=1}^{N}f_{i}(x_{i})\\
\text{subject to}\,g_{i}(x_{i})\leq0,
\end{split}\label{eq:Problem3}
\end{equation}
and the consensus {\mz sub-problem} is
\begin{equation}
\underset{t\rightarrow\infty}{\text{lim}}\,(x_{i}-x_{j})=0,\,\,\,\,i,j=1,\ldots,N.\label{eq:consensusProb}
\end{equation}
{\crg Before proceeding to solving the above mentioned { sub-problem}s, i.e.
the minimization { sub-problem} \eqref{eq:Problem3} and the consensus sub-problem \eqref{eq:consensusProb},
we present some optimality conditions for the optimization { sub-problem} through the following lemma.
Later in Section \ref{sec:MR}, we will use these conditions to show the convergence of our algorithm.}
\begin{lemma}
\cite[p. 243]{boyd2004convex}\label{lem:(KKT-Conditions)} (KKT
Conditions) Given Assumption 1 and 2, $\overline{x}^{*}=\left[x_{1}\ldots x_{N}\right]^{\top}$
is the optimal solution of the problem (\ref{eq:Problem3}) if and
only if there exist Lagrangian multipliers, $\lambda_{i}^{*}>0$ ,
$i=1,\ldots,N$, such that the following conditions are satisfied
\begin{gather}
g_{i}(x_{i}^{*})\leq0,\,\,\,\,\,\,\lambda_{i}^{*}g_{i}(x_{i}^{*})=0,\,\,\,\,\,i=1,\ldots,N,\label{eq:KKT1}\\
\sum_{i}^{N}\frac{\partial f_{i}(x_{i}^{*})}{\partial x_{i}}+\lambda_{i}^{*}\frac{\partial g_{i}(x_{i}^{*})}{\partial x_{i}}=0,\,\,\,\,i=1,\ldots,N.\label{eq:KKT2}
\end{gather}
\end{lemma}

To solve the minimization { sub-problem} (\ref{eq:Problem3}), we focus
on the {\crg primal-dual} method that seeks the saddle point of the Lagrangian associated with convex
optimization {\mz sub-problem} (\ref{eq:Problem3}). {\crg The Lagrangian is defined by}

\begin{equation}
L(\bar{x},\bar{\lambda})=\sum_{i=1}^{N}f_{i}(x_{i})+\lambda_{i}g_{i}(x_{i}),\label{eq:LagrangianDef}
\end{equation}
where $\bar{\lambda}=\left[\lambda_{1}\ldots\lambda_{N}\right]^{\top}$and
$\bar{x}=\left[x_{1}\ldots x_{N}\right]^{\top}$. $L(\bar{x},\bar{\lambda})$
is convex in $\bar{x}$ and concave in
$\bar{\lambda}$. We have the following properties for $L(\bar{x},\bar{\lambda})$,

\begin{eqnarray}
L(\bar{x}^{*},\bar{\lambda}) & \geq & L(\bar{x},\bar{\lambda})+\nabla_{\bar{x}}L(\bar{x},\bar{\lambda})^{\top}.(\bar{x}^{*}-\bar{x}),\label{eq:LagProp1}\\
L(\bar{x},\bar{\lambda}^{*}) & \leq & L(\bar{x},\bar{\lambda})+\nabla_{\bar{\lambda}}L(\bar{x},\bar{\lambda})^{\top}.(\bar{\lambda}^{*}-\bar{\lambda}),\label{eq:LagProp2}
\end{eqnarray}
where $\left(\bar{x}^{*},\bar{\lambda}^{*}\right)$ is
said to be the saddle point of $L(\bar{x},\bar{\lambda})$ \cite[p. 238]{boyd2004convex}.
The following inequalities hold for all $(\bar{x},\bar{\lambda})\in dom(L)$

\begin{equation}
L(\bar{x}^{*},\bar{\lambda})\leq L(\bar{x}^{*},\bar{\lambda}^{*})\leq L(\bar{x},\bar{\lambda}^{*}).\label{eq:SaddlePointDef}
\end{equation}
%
From (\ref{eq:LagrangianDef}), one can define the Lagrangian function
for node $i$ as
\begin{equation}
L_{i}(x_{i},\lambda_{i})=f_{i}(x_{i})+\lambda_{i}g_{i}(x_{i}).\label{eq:LagrangianForNode}
\end{equation}
In the sequel, we will use $L$ and $L_{i}$ to denote the aggregate Lagrangian
(\ref{eq:LagrangianDef}) and the Lagrangian corresponding to
node $i$, i.e. (\ref{eq:LagrangianForNode}), respectively. {Hence, the main task of this paper is to find the saddle point of \eqref{eq:LagrangianDef} while consensus on the agents' states is also achieved.}

\section{Main Results}
\label{sec:MR}
We propose the following algorithm to find the
saddle point of \eqref{eq:LagrangianDef} and satisfy the consensus constraint \eqref{eq:consensusProb}
\begin{eqnarray}
\dot{x}_{i} & = & -\alpha\nabla_{x_{i}}L_{i}+h_{i},\,\,\,\,\,\,\,i=1,\ldots,N,\label{eq:Dynamic1}\\
\dot{\lambda}_{i} & = & \left[\nabla_{\lambda_{i}}L_{i}\right]_{\lambda_{i}}^{+},\label{eq:Dynamic2}
\end{eqnarray}
where $\alpha>0$ and $h_{i}=-\beta\sum_{j\in \mathcal{N}_{i}}(x_{i}-x_{j})$
with $\beta>0$. {\crg $\mathcal{N}_{i}=\{j|j\in\mathcal{N},(j,i)\in E\}$ is the set of neighbors corresponding to node $i$.} Note that $-\alpha\nabla_{x_{i}}L_{i}+h_{i}$
acts as the control input for agent $i$, i.e.
\begin{equation}
u_{i}=-\alpha\nabla_{x_{i}}L_{i}+h_{i}.\label{eq:ConsProt}
\end{equation}
In \eqref{eq:Dynamic2}, a positive projection is used to ensure that
Lagrangian multipliers remain non-negative. For scalars, $\left[p\right]_{q}^{+}=p$
if $p>0$ or $q>0$, and $\left[p\right]_{q}^{+}=0$ otherwise.
When $\left[p\right]_{q}^{+}=0$, the projection is said to be active.
{\crb Therefore, in \eqref{eq:Dynamic2} when $\lambda_i>0$ and $g_i(x_i)<0$,
$\dot{\lambda}_i<0$ and  $\lambda_i$ decreases until it reaches $0$ where the projection
becomes active and it remains $0$ until the sign of $g_i(x_i)$ turns. {\crb Note that we start with
$\lambda_i(0)>0$; therefore, $\lambda_i\geq0$ for all $t>0$.}}
One can define the set of active projection
by $\sigma=\{i:\lambda_{i}=0,\,g_{i}(x_{i})<0\}$.
Note that the control command \eqref{eq:ConsProt}, consists of two parts. The first part is to minimize the local cost function and the second part is associated with the consensus error.
{\mz
 The following lemma is instrumental to some of the results presented in this paper.}
\begin{lemma}
\cite{horn2012matrix}\label{lem:Lemma1} (Courant-Fischer Formula) The second smallest non-zero
eigenvalue of the matrix $M\in\mathbb{R}^{N\times{N}}$, that we denote by $v_2(M)$, {\mz  satisfies
 $v_{2}(M)=\displaystyle \min_{x^{\top}\mathbf{1}_{N}=0,x\neq\mathbf{0}_{N}}\frac{x^{\top}Mx}{x^{\top}x}$.}
\end{lemma}
{\mz
Before proving that the algorithm in  (\ref{eq:Dynamic1}) and (\ref{eq:Dynamic2})
yields the saddle point of (\ref{eq:LagrangianDef}), we show that
the positions of agents, $x_{i}$, $i\in\mathcal{N}$, reach consensus,
when taking control input as $u_{i}=-\alpha\nabla_{x_{i}}L_{i}+h_{i}$,
$i\in\mathcal{N}$. This is established in the next proposition.
}
{
\begin{proposition}
\label{prop:consensustheorem}{ \crb Suppose that the graph $\mathcal{{G}}$
is connected and undirected. Then, there exists some finite $t_{1}$
such that the agents \eqref{eq:1}{ satisfy
$\left|x_{i}(t)-x_{j}(t)\right|\leq\delta_0$,
$i,j=1,\ldots,N$, with $\delta_0$ small as desired, for $t>t_{1}$ 
\eqref{eq:ConsProt}, if $\left| \nabla_{x_{i}}L_{i}-\nabla_{x_{j}}L_{j}\right| <\omega_{0}$,
$i,j=1,\ldots,N$, with $\omega_{0} \in{\mathbb{R}}^{+}$.} }\end{proposition}
\begin{proof}
{\crb
The aggregate dynamics of agents \eqref{eq:1} with \eqref{eq:ConsProt}
is $\dot{\bar{x}}=-\beta DD^{\top}\bar{x}+\Omega,\label{eq:ConsProp1}$
where $\varOmega=[-\alpha\nabla_{x_{1}}L_{1}\ldots-\alpha\nabla_{x_{N}}L_{N}]^{\top}$.
{\crb Let the network's consensus error be defined by $\bar{e}_{x}=\Pi\bar{x}$,
where $\Pi=I_{N}-\frac{1}{N}\mathbf{1}_{N}\mathbf{1}_{N}^{\top}$
and $\bar{x}$ denotes the aggregate state of the network, that is
defined by $\bar{x}=\left[x_{1}\ldots x_{N}\right]^{\top}$.
Note that $\mathbf{1}^{\top}\Pi=\mathbf{0}$ and $\Pi\mathbf{1=0}$.}
Thus, { one can write}
\begin{equation}
\dot{\bar{e}}_{x}=-\beta DD^{\top}\bar{e}_{x}+\Pi\Omega,\label{eq:ConsProp2}
\end{equation}
{ where $D=[d_{ik}]\in \mathbb{R}^{N\times|E|}$ is the  incidence
matrix associated with the topology $\mathcal{G}$. And, its entries i.e. $d_{ik}$, are  obtained by  assigning an arbitrary orientation for the
edges in $\mathcal{G}$. For instance, if one  considers the $k^{th}$ edge i.e.  $e_{k}=(i,j)$,
then  $d_{ik}=-1$ if the edge $e_{k}$ leaves node $i$, $d_{ik}=1$
if it enters node $i$, and $d_{ik}=0$ otherwise.}
We choose the Lyapunov candidate function $V(\bar{e}_{x})=\frac{1}{2}\bar{e}_{x}^{\top}\bar{e}_{x}$.
By taking time derivative from $V(\bar{e}_{x})$ along the trajectories of $\bar{e}_x$, one can write
\begin{eqnarray}
  \dot{V}(\bar{e}_{x}) &=& -\beta \bar{e}_{x}^{T} DD^{\top} \bar{e}_{x}+\bar{e}_x^\top\Pi\Omega \nonumber \\
   &\leq& -\beta {v}_{2}(DD^{\top})\left\Vert \bar{e}_{x}\right\Vert ^{2}+ \bar{e}_x^\top\Pi\Omega \nonumber \\
   &\leq& -\beta {v}_{2}(DD^{\top})\left\Vert \bar{e}_{x}\right\Vert ^{2}+ \alpha \left\Vert \bar{e}_{x}\right\Vert \omega_0\label{eq:ConsProp4}
\end{eqnarray}
{\crb where ${v}_{2}({DD^{\top}})$ denotes the smallest non-zero eigenvalue of $DD^{\top}$.}
In the above, the first inequality is resulted from Lemma \ref{lem:Lemma1},
 and the second inequality is resulted from the assumption $\left| \nabla_{x_{i}}L_{i}-\nabla_{x_{j}}L_{j}\right| <\omega_{0}$ given in the statement of the proposition.
 From \eqref{eq:ConsProp4}, one can say that{
 \begin{equation}
  \dot{V}(\bar{e}_{x})\leq-\theta\left\Vert \bar{e}_{x}\right\Vert ^{2}+\left(\left(\theta-\beta {v}_{2}(DD^\top)\right)\left\Vert \bar{e}_{x}\right\Vert+ \alpha\omega_0\right)\left\Vert \bar{e}_{x}\right\Vert,
\end{equation}}
where $0<\theta<1$. {  For $\left\Vert \bar{e}_{x}\right\Vert \geq \displaystyle \frac{\omega_{0}\alpha}{\beta {v}_{2}(DD^\top)\theta}$, we obtain $\dot{V}(\bar{e}_x)<0$.}
Now, we are ready to invoke Theorem 5.1 from \cite{khalil1996nonlinear} that guarantees that by choosing $\beta$ large
enough, one can make the consensus error, $\delta_0$, { as} small as desired.}
\end{proof}
\begin{remark}
Assumption {$\left| \nabla_{x_{i}}L_{i}-\nabla_{x_{j}}L_{j}\right| <\omega_{0}$
in Proposition \ref{prop:consensustheorem} seems to be unreasonable {\mz at the  first glance}  as  it
assumes that the primal and dual variables $x_{i}$
and $\lambda_{i}$, $i=1,\ldots,N$, must remain bounded. However, we will show by the following lemma
that this requirement always holds. { It is worthwhile  mentioning that by choosing a conservative bound on $\omega_{0}$
one can adjust the protocol's parameters to reach consensus with {\mz any} desired accuracy.}}
\end{remark}
{\mz
We now assert that the trajectories generated by the dynamics (\ref{eq:Dynamic1})
and (\ref{eq:Dynamic2}) are globally bounded.}
\begin{lemma}
\label{lem:Boundedness}Given that the graph $\mathcal{G}$ is connected and
undirected, the solutions of (\ref{eq:Dynamic1}) and (\ref{eq:Dynamic2})
are globally bounded. \end{lemma}
\begin{proof}
We study boundedness of the solutions of (\ref{eq:Dynamic1}) and (\ref{eq:Dynamic2})
by Lyapunov stability analysis. Let us define a quadratic Lyapunov function as
\begin{equation}
W(\bar{x},\bar{\lambda})=\frac{1}{2\alpha}(\bar{x}-\bar{x}^{*})^{\top}(\bar{x}-\bar{x}^{*})+\frac{1}{2}(\bar{\lambda}-\bar{\lambda}^{*})^{\top}(\bar{\lambda}-\bar{\lambda}^{*}).\label{eq:LyapFunBounded}
\end{equation}
\noindent In the above equation, $\left(\bar{x}^{*},\bar{\lambda}^{*}\right)$
represents a saddle point equilibrium associated with $L(\bar{x},\bar{\lambda})$.
By taking derivative from
both sides of (\ref{eq:LyapFunBounded}) along the trajectories (\ref{eq:Dynamic1})
and (\ref{eq:Dynamic2}), with respect to time, we will have
\begin{eqnarray}
\dot{W}(\bar{x},\bar{\lambda}) & = & 
-(\bar{x}-\bar{x}^{*})^{\top}\mathbf{\mathbf{\nabla}}_{\bar{x}}L+\frac{1}{\alpha}(\bar{x}-\bar{x}^{*})^{\top}H\nonumber \\
 &  & +\sum_{i=1}^{N}(\lambda_{i}-\lambda_{i}^{*})\left[\nabla_{\lambda_{i}}L_{i}\right]_{\lambda_{i}}^{+},\label{eq:Bounded1}
\end{eqnarray}
 where $\mathbf{\mathbf{\nabla}}_{\bar{x}}L=\left[\nabla_{x_{1}}L_{1}\ldots\nabla_{x_{N}}L_{N}\right]^{\top}$
and $H=\left[h_{1}\ldots h_{N}\right]^{\top}$.

{\crb Suppose that  for some index $i$, the projection becomes active   i.e. $i\in\sigma$. In this case $\lambda_i=0$ and  $\nabla_{\lambda_{i}}L_{i}=g_i(x_i)<0$. It is worthwhile noting that $\lambda_i<0$ never holds when parameters are initialized by positive values.
Thus, in this case one can conclude  that $(\lambda_{i}-\lambda_{i}^{*})\nabla_{\lambda_{i}}L_{i}\geq0$
due to the fact that  $\nabla_{\lambda_{i}}L_{i}<0$ and  $\lambda_{i}^{*}\geq0$.  On the other hand, for the agents the projection is not active, $(\lambda_{i}-\lambda_{i}^{*})\left[\nabla_{\lambda_{i}}L_{i}\right]_{\lambda_{i}}^{+} =$ $(\lambda_{i}-\lambda_{i}^{*})\nabla_{\lambda_{i}}L_{i}$ holds.} Thus, we can  assert that the following inequality holds.
\begin{equation}
\dot{W}(\bar{x},\bar{\lambda})\leq-(\bar{x}-\bar{x}^{*})^{\top}\mathbf{\mathbf{\nabla}}_{\bar{x}}L+\frac{1}{\alpha}(\bar{x}-\bar{x}^{*})^{\top}H+(\bar{\lambda}-\bar{\lambda}^{*})^{\top}\nabla_{\bar{\lambda}}L.\label{eq:Bounded1_1}
\end{equation}
Then, from (\ref{eq:LagProp1}) and (\ref{eq:LagProp2}), we have
\begin{align}
\dot{W}(\bar{x},\bar{\lambda}) & \leq - L(\bar{x}^{*},\bar{\lambda})+L(\bar{x},\bar{\lambda})+\frac{1}{\alpha}(\bar{x}-\bar{x}^{*})^{\top}H\nonumber \\
 &   -L(\bar{x},\bar{\lambda})+L(\bar{x},\bar{\lambda}^{*})\nonumber \\
 & \leq  \frac{1}{\alpha}(\bar{x}-\bar{x}^{*})^{\top}H\label{eq:bounded2}\\
 & = -\frac{\beta}{\alpha}\sum_{i=1}^{N}(x_{i}-x^{*})\sum_{j\in{N}_{i}}(x_{i}-x_{j})\nonumber \\
 & = -\frac{\beta}{\alpha}\sum_{i=1}^{N}x_{i}\sum_{j\in\mathcal{N}_{i}}(x_{i}-x_{j}).
\end{align}
The inequality \eqref{eq:bounded2} is due to \eqref{eq:SaddlePointDef}. Furthermore, the last equality results from the fact that $\sum_{i=1}^{N}\sum_{j\in\mathcal{N}_{i}}(x_{i}-x_{j})=0$
in a network with the undirected graph $\mathcal{{G}}$.
It is easy to show that $-\sum_{i=1}^{N}x_{i}\sum_{j\in \mathcal{N}_{i}}(x_{i}-x_{j})\leq0$
in an undirected graph. Hence, $\dot{W}(\bar{x},\bar{\lambda})\leq0$, and, thus, the proof is concluded.
\end{proof}
The dynamics \eqref{eq:Dynamic1} and \eqref{eq:Dynamic2}
can be regarded as a hybrid system {\crg due to switching projection
operator on the right side of the relation \eqref{eq:Dynamic2}}.
Thus, before proceeding to the main result of this section, we introduce the LaSalle's
invariance principle for hybrid systems through a lemma first given
in \cite{lygeros2003dynamical} and later summarized in \cite{feijer2010stability}.
\begin{lemma}
\cite{feijer2010stability}\label{lem:InvarPrinciple}
Consider the hybrid dynamics \eqref{eq:Dynamic1} and \eqref{eq:Dynamic2} with
a compact invariant set $\mathcal{O}$ and there exists a continuously differentiable
positive function ${V}(\dot{\bar{x}},\dot{\bar{\lambda}};\sigma)$ that decreases
along trajectories in $\mathcal{O}$. Then, every trajectory
generated by the hybrid dynamics
and initiated in $\mathcal{O}$ converges to $M$, the maximal invariant
set within $\mathcal{O}$, which satisfies {\mz
\begin{itemize}
	\item [a)] $\dot{V}(\dot{\bar{x}},\dot{\bar{\lambda}};\sigma)=0$ in intervals of
	fixed $\sigma$,
	\item [b)]  $V(\dot{\bar{x}},\dot{\bar{\lambda}};\sigma^{-})=V(\dot{\bar{x}},\dot{\bar{\lambda}};\sigma^{+})$
	if $\sigma$ switches between $\sigma^{-}$and $\sigma^{+}$.
\end{itemize}
}
\end{lemma}
Next, in the light of the above lemma, we express the main result of this section.
\begin{theorem}
\label{thm:Theorem1} Assume that $f_{i}(x_{i})$ and $g_{i}(x_{i})$, $i\in\mathcal{N}$, are twice continuously differentiable
on $\mathbb{R}$. Given Assumption 1 and 2, the dynamics \eqref{eq:Dynamic1}
and \eqref{eq:Dynamic2} will converge to $\left(\bar{x}^{*},\bar{\lambda}^{*}\right)$
that is the solution to the optimization {\mz sub-problem} \eqref{eq:Problem3}.
\end{theorem}
\begin{proof}
{\crb To prove the theorem, it suffices to show that dynamics (\ref{eq:Dynamic1})
and (\ref{eq:Dynamic2}) will converge to a saddle point associated
with the Lagrangian function (\ref{eq:LagrangianDef}).
{
To this end, we split the proof into two parts. We first illustrate that the Lyapunov function
\begin{equation}\label{eq:LyapFunc}
V(\dot{\bar{x}},\dot{\bar{\lambda}};\sigma)=\frac{1}{2\alpha}\sum_{i=1}^{N}\dot{x}_{i}^{2}+\frac{1}{2}\sum_{i=1,i\notin\sigma}^{N}\dot{\lambda}_{i}^{2}.
\end{equation}
is always decreasing.  Then, in the second part, we appeal to Lemma \ref{lem:InvarPrinciple}  to establish that the optimality conditions in Lemma \ref{lem:(KKT-Conditions)}  hold.
}
{  To examine the above Lyapunov function, we only need to consider two scenarios, namely, the one in which the index set  $\sigma$  changes and the other one where this set is fixed. One should note that in the former case the Lyapunov function \eqref{eq:LyapFunc} might be discontinuous as $\dot{\lambda}_{i}$ switches when $\sigma$ changes according to \eqref{eq:Dynamic2}. However, in the latter, the Lyapunov function \eqref{eq:LyapFunc} is continuous. In the following, we establish that in both cases  the positive function \eqref{eq:LyapFunc} is always non-increasing.}
We first assume that $\sigma$ is fixed. Taking derivative of $V(\dot{\bar{x}},\dot{\bar{\lambda}};\sigma)$
along the trajectories \eqref{eq:Dynamic1} and \eqref{eq:Dynamic2} with respect to time, we { obtain}
\begin{equation}
\begin{split}
\dot{V}(\dot{\bar{x}},\dot{\bar{\lambda}};\sigma)   = & 
 \sum_{i=1}^{N}\dot{x}_{i}\left(-\frac{\partial^{2}L_{i}}{\partial x_{i}^{2}}\dot{x}_{i}-\frac{\partial^{2}L_{i}}{\partial\lambda_{i}\partial x_{i}}\dot{\lambda}_{i}+\frac{\dot{h}_{i}}{\alpha}\right)  \\
   & + \sum_{i=1,i\notin\sigma}^{N}g_{i}(x_{i})\frac{\partial g_{i}(x_{i})}{\partial x_{i}}\dot{x}_{i}
\end{split}
\end{equation}
 { The above equations can be simplified by expanding some of its terms into two cases, namely, $i \in \sigma$ and $i \notin \sigma$. Note that when $i\in\sigma$, $\lambda_{i}=0$, 	$\dot{\lambda}_{i}=0$. Thus, we can write}
\begin{align}
\dot{V}(\dot{\bar{x}},\dot{\bar{\lambda}};\sigma) & =
-\sum_{i=1}^{N}\dot{x}_{i}\left(\frac{\partial^{2}f_{i}(x_{i})}{\partial x_{i}^{2}}\dot{x}_{i}-\frac{\dot{h}_{i}}{\alpha}\right)\nonumber \\
  &   -\sum_{i=1,i\notin\sigma}^{N}\dot{x}_{i}\left(\lambda_{i}\frac{\partial^{2}g_{i}(x_{i})}{\partial x_{i}^{2}}\dot{x_{i}}+\frac{\partial g_{i}(x_{i})}{\partial x_{i}}g_{i}(x_{i})\right)\nonumber \\
 &   +\sum_{i=1,i\notin\sigma}^{N}g_{i}(x_{i})\frac{\partial g_{i}(x_{i})}{\partial x_{i}}\dot{x}_{i}.\label{eq:TheoremPrf2}
\end{align}
{ Then after a simple algebraic simplification, it is easy to verify that}
\begin{eqnarray}
\dot{V}(\dot{\bar{x}},\dot{\bar{\lambda}};\sigma) & = & -\sum_{i=1}^{N}\dot{x}_{i}^{2}\frac{\partial^{2}f_{i}(x_{i})}{\partial x_{i}^{2}}-\sum_{i=1,i\notin\sigma}^{N}\lambda_{i}\dot{x}_{i}^{2}\frac{\partial^{2}g_{i}(x_{i})}{\partial x_{i}^{2}}\nonumber \\ & & +\frac{1}{\alpha}\sum_{i=1}^{N}\dot{x}_{i}\dot{h}_{i}.\label{eq:TheoremPrf3}
\end{eqnarray}
{ From the definition of $h_{i}$, we attain the following equality.}
\begin{equation}
\sum_{i=1}^{N}\dot{x}_{i}\dot{h}_{i} = -\frac{\beta}{2}\sum_{i=1}^{N}\sum_{j\in \mathcal{N}_{i}}\left(\dot{x}_{i}-\dot{x}_{j}\right)^{2}. \label{eq:TheoremPrf4}
\end{equation}
Now, with substituting \eqref{eq:TheoremPrf4} in \eqref{eq:TheoremPrf3},
we obtain
\begin{eqnarray}
\dot{V}(\dot{\bar{x}},\dot{\bar{\lambda}};\sigma) & = & -\sum_{i=1}^{N}\dot{x}_{i}^{2}\frac{\partial^{2}f_{i}(x_{i})}{\partial x_{i}^{2}}-\sum_{i=1,i\notin\sigma}^{N}\lambda_{i}\dot{x}_{i}^{2}\frac{\partial^{2}g_{i}(x_{i})}{\partial x_{i}^{2}}\nonumber \\
& & -\frac{\beta}{2\alpha}\sum_{i=1}^{N}\sum_{j\in \mathcal{N}_{i}}\left(\dot{x}_{i}-\dot{x}_{j}\right)^{2}\label{eq:TheoremPrf5}
\end{eqnarray}
From Assumption 1 and that $\frac{\beta}{\alpha}>0$, it is
attained that 
\begin{equation}
\dot{V}(\dot{\bar{x}},\dot{\bar{\lambda}};\sigma)\leq0.\label{eq:LyapunovIneq}
\end{equation}
Hence, ${V}(\dot{\bar{x}},\dot{\bar{\lambda}};\sigma)$ is non-increasing when $\sigma$ does not change. \\
In the following, we will show that the same property holds even when  the set $\sigma$ changes. Consider conditions under which the { index } set $\sigma$ varies: (1) { Consider the case  at given time index, say $t_0$, the index set $\sigma$ is enlarged. This happens when   there is a  larger number of constraints with $g_{i}(x_{i}(t_0^+))<0$ compared to { those with $g_{i}(x_{i}(t_0^-))<0$.}  We then obtain $V(t^{+})\leq V(t_0^{-})$ as $\dot{\lambda}_{i}(t_0^{+})=0$. { Here $t_0^-$ and $t_0^+$ stand for the moment just before and after $t_0$, respectively.} (2) { Now suppose that the index set $\sigma$ shrinks. This case occurs when the set loses a constraint $i$ at time
$t_0$ and  $g_{i}(x_{i}(t_0^+))$ becomes positive}. }Since $g_{i}\left(\cdot\right)$ is a continuous function
and $x_{i}$ is continuous as well, it can be said that { this function}
has passed through zero to become positive.  The latter supports that the new term ${\dot{\lambda}_{i}}^{2}$
is added to $V(\dot{\bar{x}},\dot{\bar{\lambda}};\sigma)$ but since {  $g_{i}(x_{i}(t_0^{+}))=g_{i}(x_{i}(t_0^{-}))$, no discontinuity}
happens. Therefore, one can say $V(\dot{\bar{x}},\dot{\bar{\lambda}};\sigma)$ does not change in this case and, therefore, remains non-increasing { according to \eqref{eq:LyapunovIneq}.}\\
Now, we invoke Lemma \ref{lem:InvarPrinciple}
that presents LaSalle’s invariance principle for hybrid systems. 
{From Lemma \ref{lem:Boundedness}, we conclude that whole space $\mathbb{R}^{2N}$ represents an invariant set for the hybrid dynamics \eqref{eq:Dynamic1} and \eqref{eq:Dynamic2}.  On the other hand, in the { first part of } the proof, we showed that the Lyapunov function \eqref{eq:LyapFunc} decreases along the trajectories produced by \eqref{eq:Dynamic1} and \eqref{eq:Dynamic2}.} According to the statement of Lemma \ref{lem:InvarPrinciple} there should exist maximal invariant set, say $M$, that satisfies conditions (a) and (b) { stated in Lemma \ref{lem:InvarPrinciple}}. In the sequel, we will show that
 \eqref{eq:Dynamic1} and \eqref{eq:Dynamic2} will stabilize at the point in which  conditions (a) and (b) are met; moreover, the KKT conditions \eqref{eq:KKT1} and \eqref{eq:KKT2} are { also} fulfilled.

We first attend to { part} (a). From { the equation} \eqref{eq:TheoremPrf5}, we {  attain }
$\dot{x}_{i}=0$, $i\in\mathcal{N}$, i.e. $\bar{x}\equiv\bar{x}^{*}$
since one can derive from \eqref{eq:Dynamic1} that $\bar{x}$ is continuous.
Also, $\sum_{i=1}^{N}\dot{x}_{i}=-\alpha\sum_{i=1}^{N}\nabla_{x_{i}}L_{i}$.
So, $\sum_{i=1}^{N}\nabla_{x_{i}}L_{i}=0$. Thereby, \eqref{eq:KKT2}
is satisfied.

As for $\bar{\lambda}$, assume that $g_{i}(x_{i}^{*})>0$, then,
$\lambda_{i}$ will grow unboundedly that it contradicts its boundedness
shown earlier in Lemma \ref{lem:Boundedness}. Therefore, $g_{i}(x_{i}^{*})\leq0$,
then two possible cases happen: i) $\lambda_{i}$ would decrease until it reaches at zero, producing
a discontinuity once the projection becomes active. This would contradict with part (b) of Lemma
 \ref{lem:InvarPrinciple}. ii) $\lambda_{i}=0$; the corresponding
projection is active for some $i$. Thus, $g_{i}(x_{i}^{*})\leq0$ and
$\lambda_{i}^{*}=0$ always hold, and, \eqref{eq:KKT1} is met.

In the above, we showed that the equilibrium point of the dynamics
\eqref{eq:Dynamic1} and \eqref{eq:Dynamic2} is a saddle point of
the Lagrangian function \eqref{eq:LagrangianDef}, and in the light of
Saddle Point Theorem \cite[Theorem 4.7]{ruszczynski2006nonlinear}, it is the optimal solution to \eqref{eq:Problem3}.}
\end{proof}

One should note that through Proposition \ref{prop:consensustheorem}, we showed consensus
on states, i.e. $x_{i}=x_{j}$, $i,j=1,\ldots,N$. Furthermore,
by Theorem \ref{thm:Theorem1}, we proved that the control inputs
\eqref{eq:ConsProt} drive the agents towards the saddle point of  the Lagrangian associated with \eqref{eq:Problem3}.
{Hence, the optimal consensus problem \eqref{eq:3} associated with the
network of single-integrator agents \eqref{eq:1} is resolved. }
\begin{remark}
There exists a trade-off between size of the control command  and permitted consensus error when selecting parameters $\alpha$
and $\beta$. As $\beta$ increases, according to Proposition \ref{prop:consensustheorem},
the  consensus error becomes smaller while the control input
size  attains a larger  value. On the other hand, with small $\alpha$, the consensus
error  decreases; however, this  decelerates the optimization process.
\end{remark}

\section{Simulation Example}\label{sec:SE}

{As mentioned earlier, results of the current paper also hold when agents are modeled by several integrators i.e. $x_i \in\mathbb{R}^m$. We exploit this fact and consider the following scenario that clearly exhibits the results of this paper through a numerical simulation.}
{Consider four agents that move in a 2-D space and are connected under a ring topology. Assume that each agent is modeled by one single-integrator dynamics per coordinate. Their local objective functions are as $f_{1}(x_{11},x_{12}) =  x_{11}^{2}+x_{12}^{2},f_{2}(x_{21},x_{22}) = (x_{21}-4)^{2}+(x_{22}-2)^{2},f_{3}(x_{31},x_{32})  =  (x_{31}-3)^{2}+4(x_{32}-1)^{2},f_{4}(x_{41},x_{42})  =  (x_{41}-1)^{2}$ Agent 1 has a local constraint as $g_{1}(x_{11},x_{12})=-x_{11}-x_{12}+1\leq0$.
Agent 2 suffers the constraint $g_{2}(x_{21},x_{22})=x_{21}^{2}+x_{22}^{2}-2\leq0$.
Agent 3 has the local constraint $g_{3}(x_{31},x_{32})=x_{31}^{2}+x_{32}^{2}-1\leq0$,
while agent 4 has no constraint. Let $\alpha=0.1$ and $\beta=10$ be the
control law's coefficients as in \eqref{eq:ConsProt}. Under the control law \eqref{eq:ConsProt}, the trajectories
of agents' positions are shown in Fig. \ref{fig:fig1} when the
initial positions of the agents 1, 2, 3, and 4 are set as $x_{1}=\left[\begin{array}{cc}
2 & 3\end{array}\right]^{\top}$, $x_{2}=\left[\begin{array}{cc}
1 & 4\end{array}\right]^{\top}$, $x_{3}=\left[\begin{array}{cc}
3 & 4\end{array}\right]^{\top}$, and $x_{4}=\left[\begin{array}{cc}
5 & 0\end{array}\right]^{\top}$, respectively. We set the initial values for the Lagrangian multipliers
as zero. The optimal solution to the problem is $\left[\begin{array}{cc}
0.85 & 0.53\end{array}\right]^{\top}$.\\
{
{Among  many existing penalty-based algorithm , due to the page limitation, we
 only compare our result with that of the algorithm  proposed by \cite{yi2015distributed} on the above example (see Fig. 2). As it is observed, with the primal-dual dynamics proposed in \cite{yi2015distributed}, the agents spiral
around the optimal point in a perplexed way to reach the optimal point. Such trajectories towards the final point will impose too much energy consumption
and practically are not feasible to achieve.}

%

\section{Conclusion}
\label{sec:Conc}
{We studied constrained optimal consensus problem for an undirected network of single-integrator
agents.} We proposed a fusion algorithm in which: i) a primal-dual gradient method
was used to satisfy KKT conditions for constrained convex optimization
problems, and ii) a consensus protocol was adopted to make all agents
reach the agreed optimal value. Then, through the theory of stability of perturbed
systems, we showed that this algorithm {\mz   delivers} consensus. Moreover,
we proved that the equilibrium point of the network's dynamics {\mz
coincides} with the optimal solution to the optimization problem, adopting
the LaSalle's invariance principle for hybrid systems. Finally,
we illustrated the performance of our proposed algorithm through a
numerical example.

 \begin{figure}
\begin{center}
\resizebox*{9cm}{!}{\includegraphics{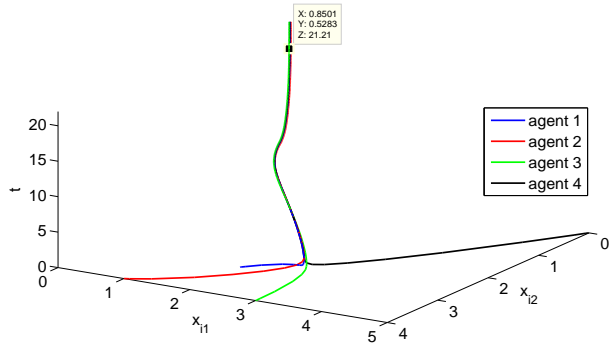}}\hspace{2pt}
\caption{\textcolor{black}{States' trajectories for a ring network of single-integrator agents under the control law \eqref{eq:ConsProt}}}
\label{fig:fig1}
\end{center}
\end{figure}

\begin{figure}
\begin{center}
\resizebox*{9cm}{!}{\includegraphics{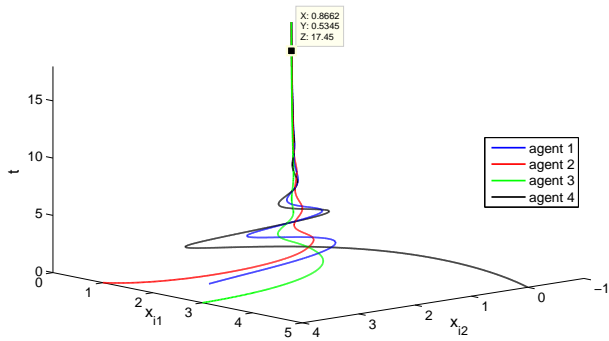}}\hspace{2pt}   %
\caption{\textcolor{black}{ Simulation results of \cite{yi2015distributed}: States' trajectories for a ring network of single-integrator agents}}
\label{fig:fig2}
\end{center}
\end{figure}

%
%
%


\bibliographystyle{spmpsci}      
\bibliography{IEEEfull}   


\end{document}